\documentclass{article}
\usepackage[english]{babel}
\usepackage {amsmath}
\usepackage {amsthm}
 \usepackage {latexsym}
\usepackage {verbatim}
\usepackage {amsfonts}

\begin{document}

\title{ Results on the existence of the Yamabe minimizer of $M^m\times \mathbf{R}^n$}
\author{Juan Miguel Ruiz \footnote{The author is supported by CONACYT}}
\date{}
\maketitle
\theoremstyle{definition} \newtheorem{lemma}{Lemma}
\theoremstyle{definition} \newtheorem{thm}[lemma]{Theorem}
\theoremstyle{definition} \newtheorem{prop}[lemma]{Proposition}
\theoremstyle{definition} \newtheorem{rem}[lemma]{Remark}
\theoremstyle{definition} \newtheorem{cor}[lemma]{Corollary}

\begin{abstract}
We let $(M^m, g)$ be a closed smooth Riemannian manifold ($m\geq2$) with positive scalar curvature $S_{g}$, and prove that the Yamabe constant of $(M \times \mathbf{R}^n,g+g_E)$ is achieved by a metric in the conformal class of $(g+g_E)$, where $g_E$ is the Euclidean metric. We also show that the Yamabe quotient of $(M \times \mathbf{R}^n,g+g_E)$ is improved by Steiner symmetrization with respect to $M$. It follows from this last assertion that the dependence on $\mathbf{R}^n$ of the Yamabe minimizer of $(M \times \mathbf{R}^n,g+g_E)$ is radial.
\end{abstract}




\section{ Introduction}

Let $(M^m,g_M)$ be a smooth closed Riemannian manifold (smooth compact manifold without boundary). Let $[g_M]$ denote the conformal class of the metric $g_M$.  The Yamabe constant of the conformal class $(M^m,[g_M])$ is defined as the infimum of the normalized total scalar curvature restricted to $[g_M]$,
\begin{equation}
\label{Yamabe}
Y(M,[g_M])= \inf_{h \in [g_M]} \frac{\int s_{h} dV_{h}}{Vol(M,h)^{\frac{m-2}{m}}},
\end{equation}
\noindent where  $S_{h}$ and  $dV_{h}$ are the scalar curvature and the volume element  of  $h$. By writing $h=f^{\frac{4}{m-2}} g_M$ ($f$ positive and $C^{\infty}$), we can rewrite (\ref{Yamabe}) in terms of functions of the sobolev space $L_1^2(M)$,

\[Y(M,[g_M])=\inf_{f \in L_1^2(M), f\neq0} Q_{g_M}(f)\]
\begin{equation}
\label{Yamabe2}
:= \inf_{f \in L_1^2(M), f\neq0 } \frac{a \int_M |\nabla f|^2 dV_{g_M}+ \int_M s_{g_M} f^2 dV_{g_M}}{\int_M f^p dV_{g_M})^{\frac{2}{p}}},
\end{equation}

\noindent where $a=\frac{4(m-1)}{(m-2)}$, $p=\frac{2m}{m-2}$. $Q_g(f)$ is called the Yamabe quotient. It is a fundamental result, proven in many steps by H. Yamabe \cite{Yam}, N. Trudinger \cite{Tru}, T. Aubin \cite{Aubin} and R. Schoen \cite{Schoen}, that, for closed manifolds, in each conformal class the infimum is achieved. The metric in each conformal class that achieves the infimum in (\ref{Yamabe}) is called a Yamabe metric and has constant scalar curvature. Meanwhile, the function that achieves the infimum in (\ref{Yamabe2}) is called a Yamabe minimizer. For any conformal class, the Yamabe constant is bounded from above, $Y(M^m,[g_M])\leq Y(S^m,[g_0])=m(m-1)Vol(S^m)$, where $g_0$ is the round metric on $S^m$ with constant sectional curvature 1. The Yamabe invariant  $Y(M)$ of $M$ is defined as the supremum of the Yamabe constants over all the conformal classes (cf. in \cite{Koba}, \cite{Schoen2}). Hence, it is an easy consequence that $Y(M) \leq Y(S^m)=Y(S^m,[g_0])$. In the following, we will denote $Y(S^m)=Y(S^m,[g_0])$ by $Y_m$.

When the Yamabe constant is non-positive, there is only one metric with constant scalar curvature in the conformal class. On the other hand, when the Yamabe constant is positive, there may be several metrics of constant scalar curvature. Examples include $(S^k\times M^m, g_0+g_M)$, where $M^m$ is a Riemannian manifold of constant scalar curvature $s_{g_M}$: it has been shown in \cite{Pet2} that in this case, the number of unit volume non-isometric metrics of constant scalar curvature in the conformal classs of $[g_0+g_M]$ grows at least linearly with $\sqrt{s_{g_M}}$. Possibly the simplest example of several metrics of constant scalar curvature, is the one exhibited in \cite{Aku}: if $(M_1^m,g_1)$ and $(M_2^n,g_2)$ are  Riemannian manifolds with constant scalar curvature, and $s_{g_1}>0$, then $\delta^n g_1+\delta^{-m} g_2$ has volume one and constant scalar curvature greater than $Y_{m+n}$.

Through the study of these cases, Akutagawa, Florit and Petean, found that if $(M_1^m,g_1)$ is a closed manifold ($m\geq2$) of positive scalar curvature and $(M_2^n,g_2)$ any closed manifold, then 
\begin{equation}
\label{akutagawa}
\lim_{r\rightarrow \infty} Y(M\times N,[g_1+r g_2])=Y(M\times \mathbf{R}^n,[g+g_E])
\end{equation}
\noindent where $g_E$ is the Euclidean metric on $\mathbf{R}^n$ (Theorem 1.1 in \cite{Aku}). Making thus the Yamabe constant $Y(M\times \mathbf{R}^n,[g+g_E])$ of high relevance in the study of the Yamabe constant of product manifolds, since, for instance, from (\ref{akutagawa}) follows that the Yamabe invariant of $M\times N$ is bounded below,

\[Y(M\times \mathbf{R}^n,[g_1+g_E])\leq Y(M\times N).\]

 As another example, the Yamabe constant of $Y(S^m\times \mathbf{R}^n,[g+g_E])$ is involved in a surgery formula for the  Yamabe invariant of a compact manifold, as have shown recent results of B. Ammann, M. Dahl and E. Humbert \cite{Amm}.

Also, it was through the case where $n=1$ that J. Petean found a lower bound to the Yamabe invariant of $M_1^m\times S^1$ (when $M_1^m$ an Einstein manifold), among other interesting results involving  $Y(M\times \mathbf{R}^n,[g+g_E])$, \cite{Pet1}.  

In this article we study the Yamabe constant $Y(M^m\times \mathbf{R}^n,[g+g_E])$, where $M^m$ ($m\geq2$), as in (\ref{akutagawa}), is a closed manifold with positive scalar curvature. 

The Yamabe problem for non-compact manifolds has not been solved completely yet. Different counter-examples and conditions for existence and nonexistence of a constant scalar curvature in the conformal class of a metric,  have been published for non-compact manifolds (cf. in \cite{Zhang2}). Results include, e.g., those of K. Akutagawa and B. Botvinnik in \cite{Bot}, where they study  complete manifolds with cylindrical ends and  solve affirmatively the Yamabe problem on cylindrical manifolds. Results include also some cases for noncompact complete manifolds of positive scalar curvature. We cite here the work of S. Kim in \cite{Kim}, where he introduces the notation

\[\mathbf{Q(M)}:= \inf_{u\in C_0^{\infty}(M)} \frac{\int_M |\nabla f|^2 dV_{g_M}+ 1/a \int_M s_{g_M} f^2 dV_{g_M}}{\int_M f^p dV_{g_M})^{\frac{2}{p}}},\]
\noindent and
\[\mathbf{\bar Q(M)}:=\inf_{u\in C_0^{\infty}(M\setminus B_r)} \frac{\int_M |\nabla f|^2 dV_{g_M}+ 1/a \int_M s_{g_M} f^2 dV_{g_M}}{\int_M f^p dV_{g_M})^{\frac{2}{p}}}\]

\noindent (where $r$ is the distance from $x$ to a fixed point $x_0\in M$, and $B_r$ the ball of radius $r$ and centered at $x_0$), and then proves the existence of a constant scalar curvature in the conformal class of $(M,g_M)$ whenever $Q(M)<\bar{Q}(M)$. 

In our case, given some of the particularities of $(M^m\times \mathbf{R}^n,g_M+g_E)$, we use a more direct approach to prove existence of a Yamabe minimizer. We first show that the Steiner symmetrization of a function ``improve'' the Yamabe quotient, making thus, the Steiner symmetrized functions, the best candidates for the Yamabe minimizer. Then, along with this result, we use the fact that $Y((M^m\times \mathbf{R}^n,g_M+g_E)<Y_{m+n}$ (a known result of Akutagawa, Florit and Petean (\cite{Aku})) to prove that the Yamabe minimizer exists and is positive and $C^{\infty}$.

 The fact that that Steiner symmetrizations ``improve'' the Yamabe quotient is a consequence of the following.
 
\begin{thm} \label{8.2}
Let $(N,g)=(M^m\times \mathbf{R}^n,g_M+g_E)$, and $u \in L_{1+}^{s}(N)$, $(1<s<\infty)$. Let $u^*$ be the Steiner symmetrization of $u$, with respect to $M$. Then $u^* \in L_{1+}^{s}(N)$, and
\end{thm}
\begin{equation}
\label{polya} 
||\nabla u^*||_s\leq ||\nabla u||_s.
\end{equation}

Indeed, using inequality (\ref{polya}) from the preceding theorem, and the fact that the norm is preserved under Steiner symetrizations ($||u^*||_s =||u||_s$, for any $s$), the next corollary follows.

\begin{cor}
\label{radial}
Consider $(N,g)=(M^m\times \mathbf{R}^n,g_M+g_E)$, and the Yamabe quotient for $2\leq s\leq p$:
\[Q_s(u)=\frac{a\int_N |\nabla u|^2 dV_g+\int_N s_g u^2 dV_g}{(\int_N u^s dV_g)^{\frac{2}{s}}},\]

\noindent then $Q_s(u^*) \leq Q_s(u)$.
\end{cor}

The main result of this paper, the existence of the Yamabe minimizer of $(N,g)=(M^m\times \mathbf{R}^n,g_M+g_E)$, is stated in the next Theorem.

\begin{thm} 
\label{principal}
Let $(N,g)=(M^m\times \mathbf{R}^n,g_M+g_E)$, with $m\geq2$ and $s_{g_M}>0$. The Yamabe minimizer of $(N,g)$ exists, and is positive and $C^{\infty}$.
\end{thm}

The result we give is sharp, since counter-examples for manifolds of the type $(M^m\times \mathbf{R}^n,g_M+g_E)$, where $M^m$ has non-positive scalar curvature or where $m<2$, and a positive Yamabe metric is not achieved are known to exist. An example of the former is $N^9=S^1\times S^1\times S^1\times S^1\times S^1\times S^1\times \mathbf{R}^3$ (with the metric  being the product of those usual ones on $\mathbf{R}^3$ and $S^1$), while an example of the latter is $N^4=S^1 \times \mathbf{R}^3$ (with the metric being the product of the usual metrics), as is shown by Zhang in \cite{Zhang2} and \cite{Zhang}.

This paper is organized as follows. In section 2 we give the precise definition of the Steiner symmetrization of a function $u$, $u^*$, with respect to $M$; we also give the definition of Polarizations and we introduce other preliminaries. We also give a proof of Theorem \ref{8.2}; many of the proofs and lemmas we give there are due to Brock and Solynin \cite{Brock} and to Jean Van Schaftingen \cite{Van}, with some minor modifications. Finally, in section 4, we give the proof of Theorem \ref{principal}. In this last section we follow the ideas of the classical proof of the Yamabe problem for compact manifolds (cf. in \cite{Lee}), and we take into account the non-compactness of the situation through the techniques of the Compactness Concentration Principle of Lions \cite{Lions1}, \cite{Lions2}.

\noindent \textbf{Acknowledgment.} The author would like to thank his supervisor J. Petean for many useful observations and valuable conversations on the subject.

\section{Proof of Theorem 1}

In this section we state some preliminary definitions and results we will need for the proof of Theorem \ref{8.2}. We begin by stating the definitions in $(N,g)=\left(M^m\times\mathbf{R}^n,g_M+g_E\right)$ of a Steiner symmetrization with respect to $M$, and of a Polarization by a polarizer in $\mathbf{R}^n$. We then prove some properties of Polarizations, such as  the fact that Polarizations preserve the $s$ norm, for any $s\leq 1$, 
\begin{equation}
\label{nablilla}
||\nabla u^H||_p=||\nabla u||_p,
\end{equation}

\noindent (lemma \ref{gradlemma}). At the end of this section we give a proof of Theorem \ref{8.2}, by showing that we can approximate any Steiner symmetrization by constructing a carefully chosen sequence of Polarizations, and then by verifying that a less or equal than relation in (\ref{nablilla}) between the gradient of $u^*$, and the gradient of $u$, is preserved in the limit of the sequence. These results are a more or less direct adaptation to our case of the work of Brock and Solynin \cite{Brock} and of Jean Van Schaftingen \cite{Van}.


\subsection{Steiner symmetrizations}

 Consider $(N,g)=\left(M^m\times\mathbf{R}^n,g_M+g_E\right)$, where $(M^m,g_M)$ is a closed Riemannian manifold and $g_E$ the Euclidean metric.  Through the course of this article we will refer to Steiner symmetrizations in $(N,g)$ with respect to $M$, simply as Steiner symmetrizations. We first define Steiner symmetrizations for sets. Let $U$ be a measurable set in $(N,g)$, we define its Steiner symmetrization $U^*$  as follows.

For each $x_0 \in M$, if 
\[Vol(U \cap (\{x_0\}\times \mathbf{R}^n),g_E) >0,\] then

\begin{equation}
\label{choose}
\left(U^*\cap(\{x_0\}\times \mathbf{R}^n)\right)=\bigg\{ 
\begin{array}{l}
  \{x_0\}\times B_{\rho}(0),\ \ \mathit{if} \ \ U  \ \ \mathit{ is\ \  open},\\
\{x_0\}\times \bar{B}_{\rho}(0), \ \ \mathit{if} \ \ U  \ \ \mathit{ is\ \ compact},
 \end{array}
\end{equation}

\noindent where $B_{\rho}(0)$ is an open ball in $\mathbf{R}^n$, of radius $\rho>0$, centered at the origin, and $\rho$ is such that
\[Vol(U \cap (\{x_0\}\times \mathbf{R}^n),g_E)= Vol(B_{0}(\rho),g_E).\]

\noindent In particular, $\rho$ depends on $x_0$.

On the other hand, if $U$ is measurable but neither open nor compact, then  the sets $U^*\cap(\{x_0\}\times \mathbf{R}^n)$ are defined in almost everywhere sense by either one of (\ref{choose}). Finally, if $Vol(U \cap (\{x_0\}\times \mathbf{R}^n),g_E) = 0$, then $U^*\cap(\{x_0\}\times \mathbf{R}^n)$ is either empty or the point $(x_0,0)$, according to whether $U\cap(\{x_0\}\times \mathbf{R}^n)$ is empty or not.

It is not hard to see that for any sets $A$, $B$ $\subset N$,
\begin{equation}
\label{nonexp}
A\subset B \Rightarrow A^* \subset B^*
\end{equation}

\noindent and that for measurable subsets $A\subset B \subset N$,
\begin{equation}
\label{nonexp2}
Vol_g(B^*\backslash A^*)\leq Vol_g (B\backslash A).
\end{equation}   

We now define Steiner symmetrizations for functions. Consider the measurable functions $u:N \rightarrow \mathbf{R}$ for which 
\[Vol(\{x \in N| u(x)>c\})< +\infty,\]
 \noindent $\forall c>0 \inf u$ (in the following, we will denote $\{x\in N|u(x)>c\}$ by $\{u>c\}$). We will call $Sym$ this class of functions. We note that  $L^s(N)$, $L_{1}^s(N)$ and $C_{0}(N)$ are subspaces of $Sym$. The Steiner symmetrization of a measurable function $u:N \rightarrow \mathbf{R}^+$ in $Sym$ is defined as follows. Let $y \in N $, then

\[u^*(y)= \sup \{c\in \mathbf{R} | y \in \{u>c\}^*\}.\]
\noindent It follows that for any $c\in \mathbf{R}$,
 
\begin{equation}
\label{nonexp1}
\{u>c\}^*=\{u^*>c\}.
\end{equation}

One important property of Steiner symmetrizations is that they are non-expansive.

\begin{lemma}
\label{expa}
Given $1\leq s< \infty$, we have
\begin{equation}
\label{exa}
||u^*-v^*||_s\leq ||u-v||_s
\end{equation}
\end{lemma}
\begin{proof}
Recall that

\[\int_N |u-v|^s dV_g = \int_{\{ \sigma \leq \tau\}}     Vol\left(\{v>\tau\}\setminus\{vu>\sigma\}\right)\]
\[+Vol\left(\{u>\tau\}\setminus\{v>\sigma\}\right) \ \ s(s-1)|\sigma-\tau|^{s-2}          d\sigma d\tau. \]

\noindent The result of the lemma then follows from equations (\ref{nonexp2}) and (\ref{nonexp1}).

\end{proof}

\subsection{Polarizations}

Let $\Sigma$ be some $(n-1)$ dimensional affine hyperplane in $\mathbf{R}^n$. Consider $M^m\times \Sigma$ and assume that $H$ is one of the open spaces into which $N=M^m \times \mathbf{R}^n$ is subdivided by $M^m\times \Sigma$. We will call $H$ a polarizer, and denote its complement in $N$ by $H^c$. Let $\bar x$ denote the reflection in $M^m\times \Sigma$ with respect to $H$. That is, for $x=(a,b) \in M^m \times \mathbf{R}^n$, with $a\in M^m$ and $b\in  \mathbf{R}^n$,

\[\bar x =(a, b^{\Sigma}),\]
where $b^{\Sigma}$ denotes the reflection of $b\in \mathbf{R}^n$, through the hyperplane $\Sigma \subset\mathbf{R}^n$, which defines $H$.


 If $u$ is measurable, we define its polarization with respect to a polarizer $H$, $u^H$, by

\begin{equation}
u^H (x)=\bigg\{ 
  				\begin{array} {l} \max\{u(x), u(\bar{x})\} \ \ \mathit{if} \ \ x \in H,\\

           							\min \{u(x), u(\bar{x})\} \ \  \mathit{if} \ \  x \in H^c.  
					\end{array}
\end{equation}

One useful property of polarizations is that the s-norms of the gradient of a function $u \in L_1^s(M\times \mathbf{R}^n)$, do not change under polarizations, as it is shown in the next lemma.

\begin{lemma} 
\label{gradlemma}
Let $u \in L_{1+}^s(N)$, $(1\leq s \leq \infty)$, and let $H$ be some polarizer. Then $u_H \in L_1^s(N)$, and $|\nabla u|$ and  $|\nabla u^H|$ are rearrangements of each other. In particular, we have
			\begin{equation}
			\label{grads}
			||\nabla u^H||_s=||\nabla u||_s
			\end{equation} 

\end {lemma}

\begin{proof}
For the sake of simplicity, we first define the reflection of $u(x)$, and the reflection of the polarization of $u$ by $H$.  That is, let 
\begin{equation}
\begin{array}{l}
				v(x):=u(\bar x),  \\
				w(x):=u^H(\bar x),\\
				for \ \ x \in H.
				
\end{array}
\end{equation}

Next, we note that 
\[u^H(x)=max\{u(x),v(x)\}= v(x)+(u(x)-v(x))_+,\]
\noindent and that
\[w(x)=min\{u(x),v(x)\}= u(x)-(u(x)-v(x))_+,\]

\noindent for all $x\in H$. Hence, we conclude that $u^H$, $w\in L_1^1(N)$, and that

\[ \nabla u^H (x)=    \bigg\{  \begin{array}{c c}
\nabla u(x) \ \ \mathit{a.e. \ \ on} \ \  \{x \in N: u(x)>v(x)\}\cap H,\\
\nabla u(x) \ \ \mathit{a.e. \ \ on} \ \ \{x \in N: u(x)\leq v(x)\}\cap H,
\end{array} \]  

\[ \nabla w(x)  =			\bigg\{  \begin{array}{c c}	\nabla v(x) \ \ \mathit{a.e.  \ \ on} \ \ \{x \in N: u(x)>v(x)\}\cap H,\\
												\nabla u(x) \ \ \mathit{a.e. \ \ on} \ \   \{x \in N: u(x)\leq v(x)\}\cap H.
\end{array} \]  
											
Now, to prove the assertions of the lemma, we define the following regions on $N$,

\begin{equation}
\label{regions}
\begin{array}{c}
R_1 = \{x \in N: u(x)>v(x)\}\cap H, \\
R_2 = \{x \in N:  u(x)\leq v(x)\}\cap H, \\
R_3 = \{x \in N: u(x)>v(x)\}\cap H^c, \\
R_4 = \{x \in N: u(x)\leq v(x)\}\cap H^c, \\
\end{array}
\end{equation}
 												
\noindent and we observe that $u^H=u$ in $R_1$ and $R_4$. Thus, we have 
\[\int_{R_1\cup R_4}{|\nabla u^H|^s}dV_g=\int_{R_1\cup R_4}{|\nabla u|^s}dV_g.\]

We also note  that $u^H=v$ in $R_2$ and $R_3$, i.e., $\int_{R_3}{|\nabla u^H|^s}dV_g=\int_{R_2}{|\nabla u|^s}dV_g$ and $\int_{R_2}{|\nabla u^H|^s}dV_g=\int_{R_3}{|\nabla u|^s}dV_g$. And so, the assertion follows:
\[\int_{N}{|\nabla u^H|^s}dV_g=\int_{R_1\cup R_4}{|\nabla u^H|^s}dV_g+\int_{R_2}{|\nabla u^H|^s}dV_g+\int_{R_3}{|\nabla u^H|^s}dV_g\]
\[=\int_{R_1\cup R_4}{|\nabla u|^s}dV_g+\int_{R_3}{|\nabla u|^s}dV_g+\int_{R_2}{|\nabla u|^s}dV_g=\int_{N}{|\nabla u|^s}dV_g.\]

\end{proof}

\begin{rem}
\label{eich}
By following the scheme of the proof of \textit{Lemma \ref{gradlemma}}, we may also note that $||u||_s=||u^H||_s$, for any $1\leq s \leq \infty$.
\end{rem}

\begin{rem}
\label{expan}
Polarizations are non-expansive (for $u$,$v$ $\in L^s(N)$, $1 \leq s \leq \infty$, $||u^H-v^H||_s\leq ||u-v||_s$).  
\end{rem}

\subsection{Approximation of Steiner symmetrizations by Polarizations}

We will now show that any Steiner symmetrization $u^*$ of a function $u$, can be approximated by a sequence of polarizations of $u$, $\{u^{H_i}\}$. To do so, we will first show that sequences of iterated polarizations $\{u^{H_i}\}$ are sequentially compact. Then, we will construct a sequence of polarizations, and establish some conditions for the convergence of the sequence to the Steiner symmetrization of the function.

We begin this section by  joining together the concepts of Steiner symmetrizations we defined earlier, with the concepts of polarizations, to define a special set of halfspaces in $N=M\times \mathbf{R}^n$. Let $\Sigma$ be a halfspace of  $\mathbf{R}^n$, we will denote by $\mathbf{H}$ the set of all halfspaces $H$ of $N$ of the form $M \times \Sigma$, and by $\mathbf{H_0}$ the set of all halfspaces $H \in \mathbf{H}$, such that $M \times \{0\} \subset H$.

\begin{rem} \label{6.7}
It follows from the definition of a polarization, from the definition of $\mathbf{H}_0$, and from the symmetry of the Steiner symmetrization, that $(u^*)_H =u^*$, for any polarizer $H\in \mathbf{H}_0$.
\end{rem}

Another fact that makes $\mathbf{H_0}$ a special set of halfspaces, is that there is always some polarizer $H \in \mathbf{H_0}$, such that $u_H$ is strictly closer to $u^*$ than $u$ is.

\begin{lemma}
\label{closer}
Let $u \in C_{0+}(N)$. If $u \neq u^*$, then there is some polarizer $H \in \mathbf{H}_0$, such that for each  $1\leq s\leq\infty$,
\[||u^H-u^*||_s<||u-u^*||_s,\]
\noindent for $1\leq s \leq \infty$.
\end{lemma}

\begin{proof}
Since $u\neq u^*$, then there is some $c>0$, such that $\{x \in N: u(x) >c\}\Delta\{x \in N:u^*(x)>c\}\neq \phi$. So, we choose some $y \in \{x \in N:u^*(x)>c\} \backslash\{x \in N:u(x)>c\}$.

There is a polarizer $H \in H^*$, such that $y^H \in \{x \in N:u(x)>c\} \backslash \{x \in N:u^*(x)>c\}$. 

We now choose a sufficiently small neighborhood $W_0\subset H$ of $y$, so that $W_0^H \subset \{x \in N: u(x)>c\} \backslash \{x\in N : u^*(x)>c\}$. We then have 
\[u^H(x)=u(\bar x)>c\geq u^*(\bar x)\]
and
\[u^*(x)>c \geq u(x)=u^H(\bar x),\]
and so, for $s\geq1$,
\begin{equation}
\label{ineq}
|u(x)-u^*(x)|^s+|u(\bar x)-u^*(\bar x)|^s>|u^H(x)-u^*(x)|^s+|u^H(x)-u^*(\bar x)|^s.
\end{equation}
If $x \in W_0$, the corresponding inequality is non-strict. The integral inequality is obtained by integration of (\ref{ineq}) over $W_0$ and of the nonstrict inequality over $H\backslash W_0$.
\end{proof}

We now prove that for a sequence of polarizations $u_m=u^{H_1H_2...H_m}$, it suffices that the polarizers satisfy $\{H_i\}_{i\leq m} \subset \mathbf{H}_0$, for the existence of a function $f$, such that a subsequence of $\{u_m\}$ converges to $f$.





\begin{lemma}
\label{sequentially}
Let $u \in C_{0+}(N)$. Let $\{u_m\}$ be a sequence of polarizations of $u$, with its respective sequence of polarizers $\{H_m \} \subset \mathbf{H}_0$, ($u_m=u^{H_1 ...H_m}$). Then there is a function $f \in C_{0+}(N)$, and an increasing subsequence $\{u_{m_k}\}$ of $\{u_{m}\}$, such that, for each $s$, $1\leq s \leq\infty$, we have
\end{lemma}
\[\lim_{k \rightarrow\infty}||f-u_{m_k}||_s=0.\]
\begin{proof}
This lemma follows from an application of the theorem of \textit{Arzela-Ascoli} (cf. \cite{Peter}).  That is, to conclude that the sequence $\{u_m\}$ is compact, we need to prove that $\{u_m\}$ is equibounded, equicontinuous and that the supports are uniformly bounded.
\begin{enumerate}

\item Since $||u||_s=||u^H||_s$, for any polarizer $H\subset \mathbf{H}_0$ (remark \ref{eich}), it follows that  $||u||_s=||u_m||_s$ for $m=1,2,...$. Thus, the functions $u_m$ are equibounded for all $m$.
\item Let 
\[w_u(\delta)=\sup\{u(x)-u(y)|d(x,y)\leq\delta\},\]
\noindent be the modulus of continuity of a function $u$.
Let $H\subset \mathbf{H}_0$ be any polarization. We proceed to analyze the different cases.

Let $\delta>0$, and consider any ball $B_{\delta}(p)$ in the domain of $u$, such that $w_u(\delta)= 
\sup_{B_{\delta}(p)}\{u(x)-u(y)\}$. If either $B_{\delta}(p) \subset H$ or $B_p(\delta) \subset H^c$, we then have $\sup\{u^H(x)-u^H(y)|d(x,y)\leq\delta\}=\sup\{u(x)-u(y)|d(x,y)\leq\delta\}$.

If, on the other hand, $B_{\delta}(p)\cap H \neq \phi $ and $B_{\delta}(p)\cap H^c \neq \phi $, then we consider that

\[\sup_{ B_{\delta}(p)} \{u^H(x)-u^H(y)\} = \sup_{ B_{\delta}(p)} \cap (R_1\cup R4) \{u(x)-u(y)\} \leq \sup_{ B_{\delta}(p)} \{u(x)-u(y)\},\]

since $(B_{\delta}(p) \cap (R_1\cup R4))\subset (B_{\delta}(p))$.

And so, we have that $w_{u^H}(\delta) \leq w_{u}(\delta)$. Which yields, by induction, $w_{u_m} \leq w_u$. Finally, since $u \in C_0(N)$, $u$ is uniformly continuous, and then the sequence $\{u_m\}$ is equicontinuous. 
\item The fact that the supports are equibounded follows from the fact that polarizations are monotone: since $u\in C_0(N)$, there is some $R>0$, and some $p \in N$, such that $\mathit{Supp}\ \ u \subseteq B_R(p)$ , and 
\[\mathit{Supp} \ \ u\subseteq B_R(p) \Rightarrow \mathit{Supp} \ \ u^H\subseteq B_R(p)^H =  B_R(p),\]
since polarizations are monotone.

And then, by induction, $\mathit{Supp} \ \ u_m \subseteq B_R(p)$.
\end{enumerate}



We conclude by the  \textit{Arzela-Ascoli} theorem that there is some $f\in C_{0+}(N)$, such that  there is some subsequence $\{u_{m_k}\}$ of $\{u_m\}$, and that $u_{m_k} \rightarrow f$.
\end{proof}

We now construct a sequence of polarizations of $u \in C_{0+}(N)$ that will converge to $u^*$. We proceed inductively. As expected, we start with $u_0=u$. Then, to choose $H_{m+1} \in \mathbf{H_0}$, so that  $u_{m+1}=u_m^{H_{m+1}}$, we look at 

\[\alpha_m = \sup_{H \in \mathbf{H_0}} \{||u_m-u^*||_1 - ||u_m^H -u^*||_1\}.\]

\noindent By lemma \ref{closer}, we know that $\alpha_m$ is always strictly positive. Now, for some fixed $\kappa$ ($0<\kappa<1$), taking $\epsilon < \alpha_m (1-\kappa)$ we note that we can always choose $H_{m+1}\in \mathbf{H_0}$ so that, 

\[0<\alpha_m < ||u_m-u^*||_1 - ||u_m^{H_{m+1}} -u^*||_1 + \epsilon<||u_m-u^*||_1 - ||u_m^{H_{m+1}} -u^*||_1 +\alpha_m (1-\kappa).\]
\noindent Then,  it follows that

\begin{equation}
\label{kappa}
\kappa \sup_{H \in \mathbf{H_0}} \{||u_m-u^*||_1 - ||u_m^H -u^*||_1\} < ||u_m-u^*||_1 - ||u_m^{H_{m+1}} -u^*||_1. 
\end{equation}

Next, we prove that the sequence of polarizations we have just constructed converges to $u^*$.

\begin{lemma}
\label{kappalemma}
Let $u\in C_{0+}(N)$. Let $\{u_m\}$ be a sequence of iterated polarizations of $u$, with corresponding halfspaces $\{H_m\} \subset \mathbf{H_0}$ ($u_m=u^{H_1H_2...H_m}$), and suppose that the $H_m$'s are chosen so that equation (\ref{kappa}) is satisfied. Then $u_m \rightarrow u^*$ in any s-norm ($1 \leq s \leq \infty$).
\end{lemma}

\begin{proof}

It follows by lemma \ref{sequentially}, that there is some $f \in  C_0(N)$, and some subsequence $\{u_{m_k}\}$ of $\{u_{m}\}$, such that $\{u_{m_k}\}$ converges to f, for any $L^p$ norm. Now, by the lower semi-continuity of the norm,
\[||u^*-f^*||_1= \lim_{k \rightarrow \infty} ||u_{m'_k}^*-f^*||_1, \]
 and since the Steiner symmetrization is a  non-expansive rearrangement, we have
 
 \[||u_{m_k}^*-f^*||_1 \leq ||u_{m_k}-f||_1. \]
 \noindent It follows that

\[||u^*-f^*||_1= \lim_{k \rightarrow \infty} ||u_{m_k}^*-f^*||_1 \leq \lim_{k \rightarrow \infty} ||u_{m_k}-f||_1 =0, \]
\noindent that is, $f^*=u^*$. Now, polarizations are also non-expansive, then, since $m_{k+1}\geq m_k+1$, we have that

\[||u_{m_{k+1}}-u^*||_1 \leq ||u_{m_{k}+1}-u^*||_1,\]

\noindent on the other hand, by equation (\ref{kappa}), for any polarizer $H \in \mathbf{H_0}$ we have,

\[||u_{m_{k}+1}-u^*||_1 \leq ||u_{m_{k}}-u^*||_1 + \kappa(||u_{m_{k}}^H-u^*||_1-||u_{m_{k}}-u^*||_1)\]

\[=(1-\kappa )||u_{m_{k}}-u^*||_1 + \kappa ||u_{m_{k}}^H-u^*||_1 \leq ||u_{m_{k}}-u^*||_1,\]

\noindent since $\kappa ||u_{m_{k}}^H-u^*||_1 -||u_{m_{k}}-u^*||_1 \leq 0$.

Hence, making $m_k \rightarrow \infty$, we get,
\[||f-u^*||_1 \leq (1- \kappa) ||f-u^*||_1 + \kappa ||f^H -u^*||_1 \leq ||f-u^*||_1,\]
\noindent that is

\begin{equation}
\label{A} 
||f-u^*||_1=||f^H-u^*||_1.
\end{equation}
 Now, since $f^*=u^*$, then $||f-f^*||_1=||f^H-f^*||_1$.

So, we cannot have $f \neq u^*=f^*$, because then we would have 
$||f-f^*||_1>||f^{H_a}-f^*||_1$, for some $H_a$ by lemma \ref{closer}, which would contradict equation (\ref{A}). Then, we can only have  that 
$\{u_{m_k}\}$ converges to $u^*$ for any $L^s$ norm.

Finally, again by the non-expansiveness of polarizations, we note that, for any $s$,
\[\lim_{k \rightarrow \infty} ||u_{k}-u^*||_s \leq \lim_{k \rightarrow \infty} ||u_{m_k}^*-u^*||_s = 0,\]

\noindent as desired.

\end{proof}

Finally, because  $C_{0+}(N)$ is dense in $L^s_+(N)$ ($1\leq s\leq \infty$), we show that the same results of lemma \ref{kappalemma} hold for functions in $L^s_+(N)$.

\begin{lemma}
 \label{convergence}
Let $u\in L^s(N)$ ($1\leq s <\infty$). For any steiner symmetrization, there is a sequence of polarizers $\{H_m\} \subset \mathbf{H}_0$, such that the sequence 
$\{u_m\}=\{u^{H_1..H_m}\}$, converges to $u^*$ in $L^s(N)$.
\end{lemma}

\begin{proof}

First, we recall that  there is a countable subset $V \subset C_0(N)$ that is dense in $L^s(N)$. Next, we choose a sequence $\{H_m\}$, for which (\ref{kappa}) holds for all $f \in V$. Then, we take any $f \in V$, sufficiently close to $u$, $||u-f||_s<\epsilon/3$. By contraction we have,

\[||u_m-u^*||_s\leq ||u_m-f_m||_s+||f_m-f^*||_s +||f^*-u^*||_s. \]

\noindent It remains to show that the right hand side is bounded by $\epsilon$. 

First, by non-expansiveness of the polarization, we have that $||u_m - f_m||_s \leq ||u-f||_s$. Second, by non-expansiveness of the Steiner symmetrization we have $||f^*-u^*||_s\leq ||f-u||_s$. Then, since $f\in V\subset C_{0+}(N)$,  choosing $m$ sufficiently large, we have  $||f_m-f^*||_s<\epsilon/3$, and then

\[||u_m-u^*||_s\leq ||u_m-f_m||_s+||f_m-f^*||_s +||f^*-u^*||_s< \epsilon,\]
\noindent as desired.
\end{proof}

\subsection{Proof of Theorem \ref{8.2}}


We are now in position to prove Theorem \ref{8.2} and conclude that $||\nabla u^*||_s \leq ||\nabla u||_s$. 

\begin{proof}
(of Theorem \ref{8.2})

Let $u \in {L_1^2}_+(N)$, and consider the sequence $\{u_m\}$ of polarizations of $u$, given by lemma \ref{convergence}. Then, for $1<s<\infty$, 
\[\lim_{m \rightarrow \infty} ||u_{m}-u^*||_s=0.\]

Also, 	$||\nabla u^H||_s=||\nabla u||_s$, by lemma \ref{gradlemma}. Then there exists some function $f \in L_1^s(N)$, and a subsequence $\{u_{m_k}\}$ of $\{u_m\}$, such that $f$ is the weak limit of $u_{m_k}$ in ${L_1^s}_+(N)$.

That is, for any compactly supported function $\varphi \in C_0(N)$,
\[\lim_{k\rightarrow \infty}\int_N \varphi \ \ div \ \ u_{m_k} dV_g = \int_N \varphi\ \ div f dV_g,\]
\noindent and
\[\lim_{k\rightarrow \infty}\int_N \varphi \ \ div \ \ u_{m_k} dV_g= -\lim_{k\rightarrow \infty}\int_N \ \ div \varphi\ \  u_{m_k} dV_g =-\int_N div \varphi \ \ div\ \ u^* dV_g.\]
\noindent Of course, this means that $v=u^*$. Finally, we recall that for $1<s<\infty$ the s-norm is weakly lower semicontinuous, that is, since $u_{m_k} \rightharpoonup u^*$ weakly 
in $L^s(N)$, then

\[||\nabla u^*||_s \leq \liminf_{k\rightarrow \infty}||\nabla u_{m_k}||_s, \]
\noindent hence 
\[||\nabla u^*||_s \leq ||\nabla u||_s, \]

\noindent since $||\nabla u^H||_s=||\nabla u||_s$ for any $H$ (lemma \ref{gradlemma}). 

\end{proof}

\section{Proof of Theorem \ref{principal}}
Let $(N,g)= (M^m\times \mathbf{R}^n,g_M+g_E)$, where $M^m$ is a closed manifold ($m\geq2$) with positive scalar curvature, and $g_E$ is the Euclidean metric. In this section we will prove the existence of a Yamabe minimizer for $(N,g)$. The basic scheme of the proof we give is the following. We first note that the subcritical Yamabe equation for $(N,g)$,

\begin{equation}
\label{subway}
a\Delta u+ S_g u^2=\lambda_s u^s, 
\end{equation}
\noindent where $S_g$ is the scalar curvature of $(N,g)$ and $a=\frac{4(n+m-1)}{n+m-2}$, can be solved for $s<p=\frac{2(n+m)}{n+m-2}$ by a positive $C^{\infty}$ function $u_s$. We achieve this by making use of the techniques of the Yamabe problem in the compact case (cf. in \cite{Lee}), and those of the Concentration Compactness Principle of Lions, (\cite{Lions1}, \cite{Lions2}). We then find a uniform bound in $L^r(N,g)$  (for some $r>p$) for the family of solution functions $\{u_s\}$, for $s$ sufficiently close to $p$. Then, using standard regularity theory and the Sobolev Embedding Theorem, we note that the  $\{u_s\}$ are $C^{2,\alpha}$ bounded in every compact subset $K_R=M\times B_R$ of $(N,g)$, and thus that $u_s \rightarrow u$ uniformly on every compact subset $K_R$ of $N$, by the Arzela-Ascoli Theorem. As a final step, we use again the techiniques of the Concentration Compactness Principle to prove that $u_s \rightarrow u$ uniformly on all of $N$, where $u$ is a positive and $C^{\infty}$ function that solves the Yamabe equation.

\subsection{The subcritical problem for $(N,g)$}

In this section we will prove that the equation

\begin{equation}
\label{two}
a\Delta u+ S_g u^2=\lambda_s u^s,
\end{equation}

\noindent has a positive smooth solution, $u_s$, for  $s<p$ and $s$ sufficiently close to $p$.
 
Let 

\begin{equation}
\label{a}
Q_s(\varphi)= \frac{\int_N{(a|\nabla \varphi|^2+S_g \varphi^2)dV_g}}{(\int_N{\varphi^s dV_g})^{2/s}},
\end{equation}

\noindent and
\begin{equation}
\label{b}
\lambda_s= \inf \{Q_s(\varphi) | \varphi \in C_0^{\infty}(N,g)\}. 
\end{equation}

Now, fix $s<p$, and choose a minimizing sequence $\{u_i\}$ of functions in $C_0^{\infty}(N)$, such that $Q_s(u_{i}) \rightarrow \lambda_s$, and such that $||u_i||_s=1$, $\forall i$. We remark that, by  Theorem \ref{8.2}, we can choose a minimizing sequence such that $u_i=u_i^*$.

Next, we note that 
\begin{equation}
\label{firstbound}
||u_i||_{1,2} \leq C_1,
\end{equation}

\noindent where $C_1$ is some constant, independent of $i$ and $s$. To prove  (\ref{firstbound}), we start with the following.

\begin{lemma}
\label{4.3}
Consider the set $\{\lambda_s\}$, $2\leq s\leq p$, with $\lambda_s$ as defined by equation (\ref{b}). Then, $\lambda_s$ is upper semi-continuous at $p$, as a function of $s$ (for any $\epsilon>0$, there is some $\delta$ such that $\lambda_s\leq \lambda_p+\epsilon$, for all $s \in (p-\delta,p)$). 
\end{lemma}


\begin{proof}

Let $\varphi\in L_1^2(N)$. Given $s',s \leq p$, since

\[Q_s(\varphi) = \frac{a\int_N|\nabla \varphi|^2 dV_g+ \int_M S_g \varphi^2 dV_g}{||\varphi||_s^2},\]

\noindent then

\begin{equation}
\label{equal}
Q_s(\varphi)=Q_{s'}(\varphi)\frac{||\varphi||_{s'}^2}{||\varphi||_s^2}.
\end{equation}

\noindent Now, since $\lambda_p$ is an infimum,  given $\epsilon>0$ we may choose $\varphi_0$ such that 

\begin{equation}
\label{moto}
\lambda_p+\epsilon>Q_p(\varphi_0).
\end{equation}

\noindent On the other hand, by continuity of the norm, we have, for some $\delta>0$,

\[1-\epsilon\leq \frac{||\varphi_0||_p^2}{||\varphi_0||_s^2}\leq 1+\epsilon,\]
\noindent for  all $s \in (p-\delta,p+\delta)$. Hence,
\[\frac{||\varphi_0||_p^2}{||\varphi_0||_s^2} Q_p(\varphi_0) \leq Q_p(\varphi_0)(1+\epsilon),\]
\noindent for all $s \in (p-\delta,p)$. Then, taking into account equation (\ref{equal}), we have 

\[Q_s(\varphi_0)\leq Q_p(\varphi_0)(1+\epsilon),\]

\noindent and then, by (\ref{moto})
\[Q_s(\varphi_0)\leq Q_p(\varphi_0)(1+\epsilon)\leq (\lambda_p+\epsilon)(1+\epsilon). \]
\noindent Finally, since $\lambda_s<Q_s(\varphi_0)$, we have
\[\lambda_s < \lambda_p + C \epsilon+\epsilon^2.\]
\noindent for  all $s \in (p-\delta,p+\delta)$, with  $C=\lambda_p+1$.

\end{proof}


\begin {rem}
\label{AFP}
It is a recent result of Akutagawa, Florit and Petean (Theorem 1.3 in \cite{Aku})  that $\lambda_p=Y(M\times \mathbf{R}^n,g_M+g_E)<Y(S^{n+m},g_0)=Y_{m+n}$, when $M$ is closed, of positive scalar curvature and $m\geq2$. Since the inequality is strict, we may choose $\epsilon>0$ small enough so that $\lambda_p+\epsilon<c<Y_{n+m}$, for some $c \in \mathbf{R}$. 
It then follows from lemma (\ref{4.3}), that for some $\epsilon>0$ small enough, there is some $\delta$, such that

 \[\lambda_s\leq \lambda_p+\epsilon<Y_{n+m},\]
 
 \noindent  for every $s \in (p-\delta,p)$. That is 
 
 \begin{equation}
 \label{asterisco}
 \frac{\lambda_s}{Y_{n+m}}<1,
 \end{equation}
 \noindent for $s$ close enough to $p$.  
\end{rem}
 
We now go back to prove (\ref{firstbound}). We note that
 
\[ ||u_i||_{1,2}= \int_N{|\nabla u_i|^2}+ \int_N{ u_i^2} \leq \frac{\lambda_s+1}{a}+  \frac{\lambda_s+1}{\min_M \{S_g\}}\]
\[\leq \frac{Y_{n+m}+1}{a}+  \frac{Y_{n+m}+1}{\min_M \{S_g\}}= C_1, \]
\noindent for $s$ close enough to $p$, by (\ref{asterisco}). That is $\{u_i \}$ is $L_1^2$ bounded independently of $i$ and of $s$. 

It then follows from the Rellich-Kondrakov theorem (cf. in \cite{Lee}) that for every compact $K\subset N$, there is some subsequence $\{u_{i_k}\} \subset \{u_i\}$ that converges weakly in $L_1^2(K)$ and strongly in $L^s(K)$ to a function that we will denote by $u_s|_K$.

Consider now the compact subsets $K_R=M\times B_R \subset N$, and note that  since $K_R\subset K_{R'}$, for $R<R'$, then we have uniqueness of limits on each compact (because the convergence on $L^s(K_R)$ is strong for each $R$). Also, note that $N=\bigcup_i^{\infty} K_i$. Then, we have our limit function $u_s$, as a well defined function on all of  $N$ by taking $u_s=lim_{R\rightarrow \infty} u_s|_{K_R}$. 

Furthermore, on each $K_R$, by the weak convergence on $L_1^2(K_R)$ , we have 
\[{||\nabla u_s|_{K_R} ||_2^2}\leq \lim_{k\rightarrow \infty}\int_{K_R} \left\langle \nabla u_s|_{K_R},\nabla u_{i_k}\right\rangle dV_g, \] 
\noindent and this implies that
\[{||\nabla u_s|_{K_R} ||_2^2}\leq \limsup_{k\rightarrow \infty} {||(\nabla u_{i_k})|_{K_R} ||_2^2}.\]

On the other hand, by the strong convergence of $u_{i_k}$ to $u_{s}|_{K_R}$ in $L^s(K_R)$, and by H\"{o}lder's  inequality, we have

\[\int_{K_R} u_s|_{K_R}^2 dV_g= \lim_{k\rightarrow \infty}\int_{K_R} u_{i_k}^2 dV_g,\]
\noindent and so, it follows that
\begin{equation}
\label{Qs}
\int_{K_R}(a|\nabla u_{s}|_{K_R}|^2+S_g u_{s}|_{K_R}^2)dV_g\leq \limsup_{k\rightarrow \infty} \int_{K_R}(a|\nabla u_{i_k}|^2+S_g u_{i_k}^2)dV_g.
\end{equation}
 Hence, to prove that $u_s$ in fact minimizes $Q_s$ on $N$, it remains to show that $||u_s||_s=1$. To this purpose, we introduce in the following lemmas the techniques of the Concentration Compactness Principle, due to Lions \cite{Lions1}, \cite{Lions2}.

\begin{lemma}
\label{concentration}

Consider a sequence $\{\rho_k\}$ of $C^{\infty}$, non-negative functions, such that  $\rho_k=\rho_k^*$, and 

\[\int_N{\rho_k dV_g}=1.\]


\noindent Then, there exists a subsequence $\{\rho_{k_j}\} \subset \{\rho_{k}\}$, and some $\alpha$ ($0 \leq \alpha\leq 1$), such that the following is satisfied: for all $\epsilon>0$, there exists some $R_{\epsilon}$  ($0<R_{\epsilon}<\infty$), and some $j_0>1$ such that
\[ \int_{M\times B_{R_{\epsilon}}} \rho_{k_j}dVg \geq \alpha -\epsilon, \]
\noindent $\forall j>j_0$.

Furthermore, for each $R>0$, given $\epsilon>0$, there is some $j_1>1$ such that 
\[\int_{M\times B_R} \rho_{k_j}dVg \leq \alpha+\epsilon, \]
\noindent $\forall j>j_1$.

\end{lemma}

\proof

First note that since $\rho_k=\rho_k^*$ for each $k>1$, then, for each $R$ we have 
\[\sup_{y\in \mathbf{R}^n} \int_{M\times \{y+ B_R\}} \rho_k dV_g = \int_{M\times B_R} \rho_k dV_g,\]

\noindent where $B_R$ is the ball of radius $R$ centered at $0$, and $y+B_R$ the ball of radius $R$ centered at $y$. Now, consider the functions
\[Q_k(t)= \int_{M\times B_t} \rho_k dV_g. \]

\noindent It follows that for each $k$, $0\leq Q_k(t)\leq 1$. Thus, the functions $Q_k(t)$ are non-negative and uniformly bounded in $\mathbf{R}^+$. Furthermore, since the $\rho_k$ are non-negative, the functions $Q_k(t)$ are non-decreasing as fucntions of $t$. 

\noindent It follows then, from the Heine-Borel theorem, that there is a subsequence $\{Q_{k_j}\} \subset \{Q_{k}\}$, and a non-negative function $Q(t)$, such that  

\[\lim_{j\rightarrow \infty}Q_{k_j} = Q(t),\]
\noindent for each $t\geq 0$.

Now, let $\lim_{t \rightarrow \infty} Q(t)=\alpha$.  We note that, $0\leq \alpha \leq 1$. Also, since $Q(t)$ is non-decreasing, and $\lim_{t\rightarrow\infty} Q(t)= \alpha$, then, given $\epsilon>0$, we may choose some $t_{\epsilon}$ such that $Q(t_{\epsilon})>\alpha-\epsilon$. Of course this implies that 

\begin{equation}
\label{diko}
\int_{M\times B_{t_{\epsilon}}}\rho_{k_j} dV_g \geq \alpha-\epsilon,
\end{equation}
\noindent for all $j>j_0$, for $j_0$ large enough. Moreover, since $Q(t)$ is non-decreasing, for all $t>0$ we have $Q(t)\leq \alpha$. This implies that

\begin{equation}
\label{diko2}
\int_{M\times B_t}\rho_{k_j} dV_g \leq \alpha+\epsilon,
\end{equation}
\noindent for all $j>j_1$, for $j_1$ large enough.

\endproof

We now show that given $\beta \in (2,p)$ the $u_k^{\beta}$ ``concentrate'' in a compact set.

\begin{lemma}
\label{novnod}

Consider a sequence $\{u_{k}^{b_k}\}$ of $C^{\infty}$, non-negative functions  ($b_k>2, \forall k$),  such that  $u_k=u_k^*$, and

\[\int_N{u_{k}^{b_k} dV_g}=1,\]
for each $k$. Assume also that the sequence $\{u_{k}\}$ is bounded in $L_1^2(N)$.


Then,  there exists a subsequence $\{u_{k_j}\} \subset \{u_{k}\}$, such that for each $\beta$ ($\beta \in (2,p)$), we have that given $\epsilon>0$, there exists some $R_{\epsilon}$ ($0<R_{\epsilon}<\infty$), such that
\[\int_{N\setminus(M\times B_{R_{\epsilon}})}{u_{k_j}^{\beta}} dVg\leq \epsilon, \]
\noindent $\forall j>j_0$, for some $j_0>1$. 
\end{lemma}

\begin{proof} 

 Take $\rho_k=u_k^{b_k}$. Then, by Lemma \ref{concentration}, we have a subsequence $\{u_{k_j}^{b_{k_j}}\}$ of $\{u_k^{b_k}\}$, and an $\alpha$, $0\leq\alpha\leq 1$,  such that, for every $\epsilon/2>0$, there is some $R_{\epsilon/2}$, such that 
\begin{equation}
\label{ok}
\alpha-\epsilon/2 < \int_{M\times B_{R_{\epsilon/2}}}u_{j}^{b_{j}} dV_g < \alpha+\epsilon/2,
\end{equation}

\noindent for all $j>j_0$, for some $j_0$ (for simplicity, we will denote $u_{k_j}^{b_{k_j}}$ by $u_{j}^{b_{j}}$).

Also, since for every $R>0$ we have $\int_{M\times B_R} \rho_{k_j} dVg \leq \alpha+\epsilon/2$ (for $j>j_1$, $j_1$ large enough), then, it follows from (\ref{ok}) that for every  for every compact $K$, $K\subset N \setminus  M\times B_{R_{\epsilon}}$, we have

\[\int_{K}u_j^{b_j} dV_g <\epsilon,\]
for all $j>j_1$.

Now, we choose $R_0>0$ such that $Vol(M\times B_{R_{0}})\leq 1$. Then, by  H\"{o}lder's inequality, for any $y\in B_{R_\epsilon}^c$

\[ \int_{M\times \{y + B_{R_0}\}} u_j^2 dV_g \leq   \int_{M\times \{y + B_{R_0}\}} u_j^{b_j} dV_g < \epsilon.\]
\noindent Now, let $R_1=R_{\epsilon}+2R_0$, then,
\begin{equation}	
			\label{Lio}		
\sup_{\{y\in B_{R_\epsilon}^c} \int_{M\times \{y + B_{R_0}\}} u_j^2 dV_g <\epsilon.
\end{equation}
\noindent Of course, we can make $\epsilon \rightarrow 0$ by making $R_{\epsilon}$ (and thus $R_1$) go to infinity.

 We next divide the proof in cases.

\textbf{Case 1}.  The sequence $\{u_{k}\}$ is bounded in $L^{\infty}(N)$.

Let $||u_k||_{\infty} <A_{\infty}$. Also, since $u_k$ is bounded in $L^2(N)$, let $A_{1,2}$ ($1<A_{1,2}<\infty)$ be such that $||u_k||_{1,2}<A_{1,2}$. Then, we have, for all $\beta_0>2$, given any $y\in B_{R_\epsilon}^c$

\[ \int_{M\times \{y + B_{R_0}\}} u_j^{\beta_0} dV_g \leq A_{\infty}^{\beta_0-2}  \int_{M\times \{y + B_{R_0}\}} u_j^2 dV_g \]

\begin {equation}
\label{astro}
<A_{\infty}^{\beta_0-2} \epsilon.
\end{equation}

\noindent by (\ref{Lio}). Now, take $\bar\beta$, such that $\bar\beta>2$. Of course, $2<2(\bar\beta-1)<\infty$. Then, by H\"{o}lder's inequality, for any given $y\in B_{R_\epsilon}^c$

\[ \int_{M\times \{y + B_{R_0}\}} |u_j|^{\bar\beta-1}|\nabla u| dV_g \]
\[\leq  \left(\int_{M\times \{y + B_{R_0}\}} |u_j|^{2(\bar\beta-1)}dV_g\right)^{\frac{1}{2}} 
\left(\int_{M\times \{y + B_{R_0}\}} |\nabla u|^2 dV_g \right)^{\frac{1}{2}} \]

\begin{equation}
\label{Lio3}
\leq (A_{\infty}^{(\bar\beta-2)}\epsilon^{1/2})(A_{1,2}),
\end{equation}

\noindent where the last inequality follows from (\ref{astro}) and the fact that $\int_{N} |\nabla u|^2 dV_g$ is uniformly bounded by $A_{1,2}$.
Then, by the Sobolev imbedding, for any $\gamma \in (1,\frac{m+n}{m+n-1})$, there is a constant $c_0$, independent of $y$, such that  

\[	\left(	\int_{M\times \{y + B_{R_0}\}}{u_j^{\beta \gamma}dV_g}\right)^{1/\gamma} \leq c_o \int_{M\times \{y + B_{R_0}\}} {u_j^{\beta }+|\nabla (u_j)^{\beta }| dV_g}\]
\[ 	\leq c_o  \int_{M\times \{y + B_{R_0}\}} {u_j^{\beta }+\beta (u_j)^{\beta -1} |\nabla u_j| dV_g},\]

\noindent for any $y\in B_{R_\epsilon}^c$. That is, 
\[\int_{M\times \{y + B_{R_0}\}}{u_i^{\bar\beta \gamma}dV_g}	\leq C_o \left(\int_{M\times \{y + B_{R_0}\}} {u_j^{\bar\beta }+\bar\beta (u_j)^{\bar\beta -1} |\nabla u_j| dV_g}\right)^{\gamma}\]

\[	\leq C_o\left(A_{\infty}^{\bar\beta-2}\epsilon+\bar\beta A_{1,2}(A_{\infty}^{(\bar\beta-2)}\epsilon^{1/2})\right)^{\gamma-1} \left(\int_{M\times \{y + B_{R_0}\}} {u_j^{\bar\beta }+\bar\beta (u_j)^{\bar\beta -1} |\nabla u_j| dV_g}\right)\]

\[\leq C_1 \epsilon^{(\gamma-1)/2} \left(\int_{M\times \{y + B_{R_0}\}} {u_j^{\bar\beta }+\bar\beta (u_j)^{\bar\beta -1} |\nabla u_j| dV_g}\right),\]

\noindent with $C_1=C_0 A_{\infty}^{(\bar{\beta}-2)(\gamma-1)}(\bar{\beta}A_{1,2})^{\gamma-1}$, and $C_0=c_0^{\gamma}$.

We then cover $\mathbf{R}^n\setminus B_{R_1}$ with balls of radius $R_0$ in some way that any point $y \in (\mathbf{R}^n\setminus B_{R_1}) $ is not covered by more than $m$ balls ($m$ a prescribed integer). It follows that

\[\int_{N\setminus M\times B_{R_1}}{u_j^{\bar\beta \gamma}dV_g} \leq m  C_1 \epsilon^{(\gamma-1)/2}  \int_{N \setminus (M\times  B_{R_1})} {u_j^{\bar\beta }+\bar\beta (u_j)^{\bar\beta -1} |\nabla u_j| dV_g}\]
\[\leq m C_1 \epsilon^{(\gamma-1)/2}  (A_{\infty}^{\bar\beta-2}A_{1,2}+\bar\beta A_{1,2} A_{\infty}^{\bar\beta-2} A_{1,2}^{1/2}) \]
\begin{equation}
\label{gamma}
\leq (m C_0 p A_{1,2}^2 A_{\infty}^2 )\epsilon^{(\gamma-1)/2}
\end{equation}
\[\leq C_2 \epsilon^{(\gamma-1)/2}.\]

Finally, by noting that $C_2$ does not depend on $y\in \mathbf{R}^n$, we can make $\epsilon\rightarrow 0$, by making $R_{\epsilon}$ (and thus $R_1$) go to infinity. That is, given $\beta \in (2,p)$, for every $\delta>0$ we may find $R_{\delta}$ such that

\[	\int_{N\setminus M\times B_{R_{\delta}}}{u_j^{\beta }dV_g} < \delta. \]


This finishes the proof of case 1. We now remove the assumption that $u_j$ is bounded in $L^{\infty}(N)$.

\textbf{Case 2}. The sequence $\{u_j\}$ is not bounded in $L^{\infty}(N)$.

We note that for any $A>1$, the function $v_j=\min\{u_j, A\}$ is bounded in $L^{\infty}(N)$, and still satisfies the conditions needed for the previous proof, so that for any $\beta_1$ ($2<\beta_1<p$), given $\epsilon>0$, we have, by equation (\ref{gamma}), some $R_{1}>0$ such that  
\begin{equation}
\label{bone}
\int_{N\setminus M\times B_{R_1}} v_j^{\beta_1} < C_3 A^2 \epsilon.
\end{equation}

\noindent where $C_3$ is a constant that does not depend on $A$. We also have
\begin{equation}
\label{btwo}
\int_{N\setminus M\times B_{R_1}} u_j^{\beta_1} dV_g \leq \int_{N\setminus M\times B_{R_1}} v_j^{\beta_1} dV_g+ \int_{N\setminus M\times B_{R_1}} {(u_j|_{\{u_j>A\}})^{\beta_1} dV_g}.
\end{equation}

We next choose $\beta_2 \in(2,p)$, such that $\beta_2 >\beta_1$. And then,
\[ {A^{\beta_2-\beta_1}} \int_{N\setminus M\times B_{R_1}} (u_j|_{\{u_j>A\}})^{\beta_1} dV_g \leq \int_{N\setminus M\times B_{R_1}} (u_j|_{\{u_j>A\}})^{\beta_2} dV_g,\]

\noindent it follows that

\begin{equation}
\label{bthree}
\int_{N\setminus M\times B_{R_1}} (u_j|_{\{u_j>A\}})^{\beta_1} dV_g \leq \frac{K}{A^{\beta_2-\beta_1}},
\end{equation}

\noindent since $\int_N u_j^{\beta_1} dV_g<K$, for some $K>0$, because $\beta_1<p$. Hence, from (\ref{bone}), (\ref{btwo}) and (\ref{bthree}), we have
 
\begin{equation}
\label{bfour}
\int_{N\setminus M\times B_{R_1}} u_j^{\beta_1}<C_3 A^2 \epsilon + \frac{K}{A^{\beta_2-\beta_1}}.
\end{equation}

Then, given $\delta>0$, we may first choose  $A$ such that $\frac{K}{A^{\beta_2-\beta_1}}<\frac{\delta}{2}$, and then we choose $\epsilon>0$, such that $C_3 A^2 \epsilon<\frac{\delta}{2}$. Of course, for this $\epsilon$ there is some $R_1$ such that $\int_{N\setminus M\times B_{R_1}} v_j^{\beta_1} < C_3 A^2 \epsilon$, and then by equation (\ref{bfour}) we have
\begin{equation}
\label{bfive}
\int_{N\setminus M\times B_{R_1}} u_j^{\beta_1}<\delta.
\end{equation}

\noindent The conclusion of the lemma follows.

\end{proof}

We now go back to prove that $||u_s||_s=1$. By taking $b_k=s$, we note that the minimizing sequence $\{u_{k}\}$ satisfies the hypothesis of lemma  \ref{novnod}, since in its construction we assumed that the minimizing sequence was symmetrized ($u_{k}=u_{k}^*$), and that $||u_{k}||_{s}=1$, for all $k>1$. On the other hand, equation (\ref{firstbound}) showed that  $\{u_{k}\}$ was uniformly bounded in $L_1^2(N)$. Then, by taking $\beta=s<p$ in lemma  \ref{novnod}, we have that for every $\delta>0$ there is some $R_1$ such that

\[	\int_{N\setminus M\times B_{R_1}}{u_j^{s}dV_g} < \delta. \]

Of course, this implies that $\alpha=1$. That is, $||u_s||_s=1$. Then, $u_s$ is a weak solution to equation (\ref{two}). It follows from a result of N. Trudinger (Theorem 3 in \cite{Tru}) that $u_s$ is smooth, since it is a weak solution of (\ref{subway}), and from the maximum principle (cf. in \cite{Lee}) that $u_s$ is positive, since $S_g$ is positive. 

We resume in the next lemma what we have just proved.
\begin{lemma}
\label{subcritical}
For $s>2$ and close enough to $p$ (close enough so that equation (\ref{asterisco}) is satisfied),  equation (\ref{two}) has a solution $u_s$, such that $Q_s(u_s)=\lambda_s$, and $||u_s||_s=1$.
\end{lemma}


\subsection{The limit as $s\rightarrow p$}

We now investigate the limit of the functions $u_s$, as $s\rightarrow p$. We will show that the functions $u_s$ converge to a function $u$, which in turn will be the Yamabe minimizer for $(N,g)$. We will also show that $u$ is positive and $C^{\infty}$.



By lemma \ref{subcritical}, we have a family $\{u_s\}$  of functions that solve equation (\ref{two}) and such that $||u_s||_s=1$. Next, we will  prove that this family is uniformly $C^{2,\alpha}$ bounded in each compact set $M\times B_R\subset N$. We will achieve this by finding first a uniform bound for $||u_s||_r$, (for some $r>p$) and then, using standard elliptic  regularity theory, and the Sobolev Embedding Theorem, we will find our $C^{2,\alpha}$ bound. We follow the techniques of Parker and Lee, \cite{Lee}. 


We begin by proving that the functions $||u_s||$ are uniformly bounded in $L^r(N)$, for some $r>p$, as $s\rightarrow p$. 
\begin{prop}
\label{4.4}
Given the collection of functions of lemma (\ref{subcritical}), $\{u_s\}$, there are some constants $s_0<p$, $r>p$, and $C>0$, such that 
\[||u_s||_r \leq C,\] 
\noindent for all $s>s_0.$
\end{prop}
 
 \begin{proof}
 



 
 Consider the Yamabe subcritical equation (\ref{two}). Let $\delta>0$ and multiply (\ref{two}) by $u_s^{1+2\delta}$. Then, integrating over $N$, we have 
 \begin{equation}
 \label{a1}
 a\int_{N} u_s^{1+2\delta} \Delta u_s dV_g +\int_{N}  S_g u_s^{2+2\delta}  dV_g = \int_{N} \lambda_s u_s^{s+2\delta} dV_g.
 \end{equation}
 Next, by setting $w=u_s^{1+\delta}$, we get $dw=(1+\delta)u_s^\delta du_s$. And so, multiplying  both sides of (\ref{a1}), by $(1+\delta)^2$, it simplifies to
 \[\frac{1+2\delta}{(1+2\delta)^2} a \int_N |dw|^2 dV_g + = \lambda_s \int_N u_s^{s-2}w^2 dV_g - \int_N S_g w^2dV_g.\]

  \noindent Then by using  the ``integration by parts'' formula,
\[ \int_{N} \left\langle \nabla \varphi, \nabla \psi \right\rangle dV_g= \int_{N} \varphi \Delta \psi dV_g,\]
 
 \noindent (cf. in \cite{Lee}, page 42), we have
 
 \begin{equation}
 \label{b1}
 \int_{N}|dw|^2 dV_g \leq \frac{(1+\delta)^2}{1+2\delta}  \frac{\lambda_s}{a} \int_N u_s^{s-2}w^2 dV_g.
  \end{equation}

 Now, since $(N,g)$ is a complete manifold, it has bounded sectional curvature, and strictly positive injective radius, then the Sobolev Embedding Theorem holds (cf. in \cite{Au1}), that is, for any $\epsilon>0$,  there is some $C_{\epsilon}$ such that
 
 \[||w||_p^2 \leq (1+\epsilon) \frac{a}{Y_{m+n}} \int_N |dw|^2 dV_g+ C_{\epsilon} \int_N w^2dV_g,\]
 \noindent hence, by equation (\ref{b1}),
 
  \[||w||_p^2 \leq (1+\epsilon) \frac{(1+\delta)^2}{1+2\delta}  \frac{\lambda_s}{Y_{m+n}} \int_N u_s^{s-2}w^2 dV_g + C'_{\epsilon} \int_N w^2dV_g,\] 
  \noindent and so, by H\"{o}lder's  inequality,
 
  \begin{equation}
  \label{jackson}
  ||w||_p^2 \leq (1+\epsilon) \frac{(1+\delta)^2}{1+2\delta}  \frac{\lambda_s}{Y_{m+n}} ||u_s||_{(s-2)(m+n)/2}^{s-2}||w||_p^2  + C'_{\epsilon} ||w||_2^2.
 \end{equation}

 Now we recall that by remark \ref{AFP}, there is some $\delta_1>0$ such that 
 
  \[\frac{\lambda_s}{Y_{m+n}}<1,\]  
 \noindent for all $s$, $p-\delta_1 \leq s \leq p$.


On the other hand, we note that  if $p-\delta \leq s\leq p$, then   
\[0\leq s- \left((s-2)\frac{n+m}{2}\right) \leq \delta \left(\frac{n+m}{2}\right).\] 

\noindent Meanwhile, by continuity of the norm, given $\epsilon>0$, there is some $\delta_{\epsilon}>0$ such that

\[  ||u_s||_{s'}\leq ||u_s||_s + \epsilon,\]

\noindent if $|s-s'|\leq\delta_{\epsilon}$. Then, by taking $\delta_2=\delta_{\epsilon}(\frac{2}{n+m})$, we have that for $s \in (p-\delta_2,p)$,

  \begin{equation}
  \label{s-2}
  ||u_s||_{(s-2)(n+m)/2}\leq ||u_s||_s + \epsilon=1+\epsilon,
  \end{equation}
  
  \noindent since $0\leq s-\left((s-2)(n+m)/2\right) \leq \delta_{\epsilon}$.
 
 Thus, in (\ref{jackson}), we can choose $\delta$ and $\epsilon$ small enough so that the coefficient of the first term  is less than $1$ and hence, can be absorbed by the left-hand side. We note that we need $s$ to be close enough to $p$ so that both (\ref{s-2}) and (\ref{asterisco}) are satisfied.

 We then have from (\ref{jackson})
  \begin{equation}
  \label{dobleu}
 ||w||_p^2\leq C||w||_2^2. 
 \end{equation}
 
 Hence, to finish the proof, we  only need to show that 
 \[||w||_2^2=||u_s||_{2(1+\delta)}^{1+\delta},\]

\noindent is bounded independently of $s$. We proceed as follows. First, we divide the support of $u_s$ in $\Omega_s=u_s^{-1}((1,\infty))$ and $\Omega_s^c$. Then we note that since $||u_s||_s=1$, then $Vol(\Omega_s)\leq1$, independently of $s$, and hence,  by H\"{o}lder's inequality
 
 \begin{equation}
  \label{amayuscula}
  \left(\int_{\Omega_s}u_s^{2(1+\delta)}\right)^{\frac{1+\delta}{2(1+\delta)}}\leq||u_s||_{2(1+\delta)}^{1+\delta}<||u_s||_s^{1+\delta}=1.
  \end{equation}
 
 Meanwhile, outside $\Omega_s$, since $u_s<1$, then
 
 \[u_s^{2(1+\delta)}<u_s^2,\]
 \noindent and then
  \begin{equation}
  \label{bmayuscula}
  \int_{\Omega_s^c}u_s^{2(1+\delta)}<\int_{N}u_s^2<C_1,
  \end{equation}
 \noindent where $C_1$ is independent of $s$, by (\ref{firstbound}). It follows from (\ref{amayuscula}) and (\ref{bmayuscula}) that $||u_s||_{2(1+\delta)}^{1+\delta}$ is bounded uniformly. And then, from (\ref{dobleu})
 \[||w||_p=||u_s||_{p(1+\delta)}^{1+\delta}\] 
 \noindent is bounded independently of $s$.

  \end{proof}
 
 
 It follows from this $L^r$ bound that we may find a $C^{2,\alpha}$ bound for the family $\{u_s\}$ on each compact subset of $N$.
 
 \begin{lemma}
 \label{alphabound}
 For the family of solutions $\{u_s\}$ in lemma \ref{4.4}, that are bounded uniformly in $L^r(N)$, there is a $C^{2,\alpha}$ bound on each compact $M\times B_R\subset N$.
 \end{lemma}
 
 \begin{proof}
 Consider any compact subset $M\times B_R\subset N$, and take $R_0, R_1, R_2$, ($R<R_0<R_1<R_2$) large enough. Of course,  for any $r>0$, $Y(M\times B_r,g_M+g_E)\leq Y(M\times \mathbf{R}^n,g_M+g_E)<Y_{m+n}$.
Now, since $u_s\in L^r(N)$ (lemma \ref{4.4}), then by (\ref{subway}), 
\[|\Delta u_s|=|\lambda_s u_s^{s-1}- \frac{S_g}{a} u_s| \in L^q(M\times B_{R_2}),\] 
\noindent with $q=\frac{r}{s-1}$. Then, by standard elliptic regularity theory (for example, Gilbarg and Trudinger, \cite{Gil}), we have $u_s \in L_2^q(M\times B_{R_1})$. And then, from the Sobolev Embedding Theorem, $u_s \in L^{r'}(M\times B_{R_1})$, with $r'=\frac{(n+m)r}{(n+m)s-(n+m)-2r}$. Of course, $r>r'$, since $r>p=\frac{(n+m)(p-2)}{2}>\frac{(n+m)(s-2)}{2}$. By iterating this procedure we get $u_s\in L_2^q$ for all $q>1$.
 
 Then, again by the Sobolev Embedding Theorem, we have $u_s \in C^{\alpha}(M\times B_R)$ for some $\alpha>0$. Thus, using standard elliptic regularity theory one more time, we conclude that $u_s\in C^{2,\alpha}(M\times B_R)$.

 This implies that we have a uniform $C^{2,\alpha}$ bound on each compact subset $M\times B_R\subset N$.
 
  \end{proof}

 It follows now from the  the Arzela-Ascoli Theorem that we can find a subsequence  $\{u_{s_k}\}\subset\{u_s\}$ which converges to its limit $u$ on each compact subset of $(N,g)$.  From this, we can construct the limit function $u$ such that $u_{s_k}$ converges to $u$ on all of $N$. Then, using lemma \ref{novnod} we will prove that $\lim_{k\rightarrow \infty} ||u_{s_k}||_p=1$. Naturally, the limit function $u$ would be a solution to the Yamabe equation, completing thus the proof of Theorem \ref{principal}.

\begin{lemma}
\label{finaldestination}
Let $\{u_s\}$ be the sequence of functions given by lemma \ref{subcritical}, then, as $s\rightarrow p$ there is a subsequence   $\{u_{s_k}\}\subset\{u_s\}$ such that it converges to a positive, $C^{\infty}$ solution, of

\[a\Delta u+S_g u = \lambda u^{p-1},\]
\noindent with
\[||u||_p=1\]
\noindent and
\[Q_p(u)=Y(N,[g])=\lambda.\]

\end{lemma} 
\begin{proof}

By lemma \ref{alphabound} we have that the sequence $\{u_s\}$  is $C^{2,\alpha}$ uniformly bounded on each compact $M\times B_R \subset N$. Then, by the Arzela-Ascoli theorem (cf. in \cite{Peter}), this implies that for each compact $K_R=M\times B_R \subset N$, there is a subsequence $\{u_{s_k}\} \subset \{u_s\}$ such that it converges in $C^2(K_R)$ norm to a function in $C^2(K_R)$ that we will denote by $u|_K$.  Then, since $K_R\subset K_{R'}$ for $R<R'$, we have uniqueness of limits on each compact (because of the $C^2(K_R)$ convergence for each R). Also, since $N=\bigcup_i^{\infty} K_i$, then we have our limit function $u$ as a well defined function on all of $N$ by taking $u=lim_{R\rightarrow \infty} u|_{K_R}$.


We now  prove that $\lim_{k\rightarrow \infty} ||u_{s_k}||_{s_k}=||u||_p=1$. We use lemma  \ref{novnod}. First, we note that the hypothesis are satisfied by $\{u_{s_k}\}$. We already know that $u_{s_k}=u_{s_k}^*$ and that $||u_{s_k}||_{s_k}=1$, for each $k>1$. On the other hand, equation (\ref{firstbound}) shows that the $u_{s_k}$ are uniformly bounded in $L_1^2(N)$. To prove that the $u_{s_k}$ are uniformly bounded in $L^{\infty}(N)$, consider the compact set $K_1=(M\times \bar{B_1})$. We recall that $u_{s_k}\rightarrow u|_{K_1}$ on $K_1$, in $C^2$ norm. Hence, for all $k>k_1$, $k_1$ large enough,
\[\sup_{K_1} u_{s_k}\leq (\sup_{K_1} u|_{K_1} )+1,\]

\noindent Then, since $u_{s_k}=u_{s_k}^*$ for all $k>1$, we know that 
\[\sup_N u_{s_k}\leq \sup_{K_1} u_{s_k}.\]
\noindent Of course this implies that $(\sup_{K_1} u_{s_k})\leq (\sup_{K_1} u|_{K_1})+1$, and then the $u_{s_k}$ are uniformly bounded in $L^{\infty}(N)$ for all $k>k_1$.


Now, let $\beta \in (2,p)$. Let $\epsilon>0$, then, by lemma \ref{novnod}, there is some $R_{\epsilon}>0$ and some $k_2>1$ such that
\begin{equation}
\label{pre}
\int_{N \setminus (M\times B_{R_{\epsilon}})} u_{s_k}^{\beta} dV_g <\epsilon
\end{equation}
\noindent for all $k>k_2$.

On the other hand, since $u_k$ is bounded uniformly in $L^{\infty}(N)$, say  $u_k\leq A_{\infty}$ (for all $k>k_3$, for some $k_3>1$) we have

\[\int_{N\setminus (M\times B_{R_{\epsilon}})} u_{s_k}^{s_k} dV_g\leq A_{\infty}^{s_k-\beta}\int_{N\setminus (M\times B_{R_{\epsilon}})} u_{s_k}^{\beta} dV_g\] 
 \begin{equation}
\label{okmaguey}
\leq C_A \int_{N\setminus (M\times B_{R_{\epsilon}})} u_{s_k}^{\beta} dV_g \leq C_A \epsilon
\end{equation}
\noindent where $C_A$ is a constant such that $C_A= \max \{1,A_{\infty}\}$ (and of course, we have chosen $k_4$ large enough so that $s_k-\beta>0$, for all $k>k_4$). The last inequality  of (\ref{okmaguey}) is an application of (\ref{pre}). It follows from (\ref{okmaguey}) that

\[\lim_{k\rightarrow \infty} ||u_{s_k}||_{s_k} =\alpha = 1.\] 


\noindent Hence $||u||_p=1$. Of course, this implies that there is a subsequence $\{u_{s_k}\}\subset\{u_s\}$ such that it converges in $C^2$ norm to a solution $u \in C^2(N)$ that satisfies
\[a\Delta u+S_g u = \lambda u^{p-1},\]
\noindent with
\[Q_p(u)=\lambda,\]
\noindent where $\lambda=\lim_{s\rightarrow p} \lambda_s$. The  following continuity lemma  implies that $\lambda=\lambda_p=Y(N,[g])$. 

\begin{lemma}
 \label{lambdas}
Consider the set $\{\lambda_s\}$ as defined by equation \ref{b}, then $\lambda_s\rightarrow \lambda_p$ as $s\rightarrow p$. 
 \end{lemma}
 \begin{proof}
 Since $\lim_{k\rightarrow \infty} ||u_k||_p=1$, recalling that $||u_k||_{s_k}=1$ and $Q_{s_k}(u_k)=\lambda_{s_k}$, we have by (\ref{equal}), 
 \[Q_p(u_k)= \frac{\lambda_{s_k}}{||u_k||_p}.\]
 \noindent Then, for $s_k$ close enough to $p$,
\[\lambda_p\leq Q_p(u_k)=\frac{\lambda_{s_k}}{||u_k||_p}\leq \lambda_{s_k}(1+\epsilon)\leq \lambda_{s_k}+\epsilon Y_{n+m}, \]
\noindent since $\lambda_{s_k}< Y_{n+m}$, for all $s_k\leq p$, by (\ref{asterisco}).

We conclude, using lemma \ref{4.3}, that $\lambda_s\rightarrow \lambda_p$ as $s\rightarrow p$. 
 \end{proof} 

Finallly, the regularity of $u$ follows from a result of N. Trudinger (Theorem 3 in \cite{Tru}), since $u$ is an $L_1^2(N,g)$ solution of the Yamabe equation. 

On the other hand, since $S_g>0$ and $u$ is smooth, it follows from the maximum principle (cf. in \cite{Lee}) that $u$ is positive.

\end{proof}

Of course, from lemmma \ref{finaldestination}, Theorem \ref{principal} follows.

\begin{tabular}{lll}
CIMAT, Jalisco S/N, Col. Valenciana, CP 36240 Guanajuato\\
Guanajuato Mexico\\
\textit{E-mail:} miguel@cimat.mx
\end{tabular}

\end{document}